\documentclass{amsart}
\usepackage{a4wide}
\usepackage[all]{xy}

\renewcommand{\phi}{\varphi}
\newcommand{\bfa}{\mathbf{a}}

\newcommand{\bft}{\mathbf{t}}
\newcommand{\bfT}{\mathbf{T}}
\newcommand{\bfX}{\mathbf{X}}

\newcommand{\Lgal}{\tilde L}
\newcommand{\gal}{{\rm Gal}}

\newcommand{\bbQ}{\mathbb Q}

\newcommand{\bbZ}{\mathbb Z}

\newtheorem{theorem}{Theorem}[section]
\newtheorem{lemma}[theorem]{Lemma}

\theoremstyle{definition}
\newtheorem{definition}[theorem]{Definition}

\theoremstyle{remark}
\newtheorem{remark}[theorem]{Remark}
\newtheorem{problem}[theorem]{Problem}
\newtheorem{exampleplain}[theorem]{Example}

\begin{document}

\fontsize{12}{18}
\selectfont

\CompileMatrices 

\title[Characterization of Hilbertian fields]{On the characterization of Hilbertian fields}%
\author[L. Bary-Soroker]{Lior Bary-Soroker}%
\address{The Raymond and Beverly Sackler school of mathematical sciences, Tel-Aviv university}%
\email{barylior@post.tau.ac.il}
\begin{abstract}
The main goal of this work is to answer a question of D\`ebes and
Haran by relaxing the condition for Hilbertianity. Namely we prove
that for a field $K$ to be Hilbertian it suffices that $K$ has the
irreducible specialization property merely for \emph{absolutely}
irreducible polynomials.
\end{abstract}


\maketitle

\section{Introduction}
Let $K$ be a number field. Hilbert's irreducibility theorem
\cite{Hilbert1892} gives for irreducible polynomials
$f_i(T_1,\ldots,T_r,X_1,\ldots, X_s)\in K[\bfT,\bfX]$,
$i=1,\ldots,n$ and a nonzero polynomial $p(\bfT)\in K[\bfT]$ an
$r$-tuple $\bfa \in K^r$ for which $p(\bfa)\neq 0$ and all
$f_i(\bfa,\bfX)$ are irreducible in $K[\bfX]$. Actually, Hilbert's
proof shows that it suffices to have a weaker irreducible
specialization property, namely to have such an irreducible
specialization only for one irreducible $f(T,X)\in K[T,X]$,
separable in $X$ and $p(T)\in K[T]$.

A field satisfying the latter property is called
\textbf{Hilbertian}. So if $K$ is Hilbertian, then the above
stronger irreducibility specialization property holds, provided that
$s=1$ and $f_i(T_1,\ldots,T_r,X)$ is separable in $X$, for each
$i=1,\ldots,n$. Moreover to have irreducible specializations for any
$s\geq 1$ and with no separability assumption, it is sufficient and
necessary that $K$ is Hilbertian and imperfect (Uchida's Theorem
\cite{Uchida80}, see also \cite[Proposition
12.4.3]{FriedJarden2005}).

Hilbert's irreduciblity theorem has numerous applications in number
theory (see e.g.\ \cite{Serre1997}). In particular, Hilbert's
original motivation for this theorem, the inverse Galois problem
over a field $K$, which asks what finite groups occur as Galois
group over $K$. Those applications make the question what fields are
Hilbertian interesting.


The main goal of this paper is to answer a question of D\`ebes and
Haran \cite{DebesHaran1999} by proving that for a field $K$ to be
Hilbertian it suffices that $K$ has the irreducible specialization
property just for absolutely irreducible polynomials:

\begin{theorem}\label{thm:main}
Let $K$ be a field. Assume that for any absolutely irreducible
$f(T,X)\in K[T,X]$, separable in $X$ and any nonzero $p(T)\in K(T)$
there exists $a\in K$ such that $p(a)\neq 0$ and $f(a,X)$ is
irreducible. Then $K$ is Hilbertian.
\end{theorem}

This theorem is known `a priori' for special fields, namely PAC
fields \cite[Theorem~4.2]{DebesHaran1999}. A field $K$ is
\textbf{Pseudo Algebraically Closed (PAC)} if $V(K)\neq \emptyset$
for any absolutely irreducible variety $V$ defined over $K$. For PAC
fields there is a connection between group theoretic properties of
the absolute Galois group $\gal(K)$ and irreducible specializations
of polynomials. We describe this connection below.

We prove Theorem~\ref{thm:main} for an arbitrary field $K$. In fact,
the argument we are using seems simpler than the argument used in
\cite{DebesHaran1999} for the case where $K$ is PAC. 

In the rest of the introduction we describe the research that led
D\`ebes-Haran to their question and then briefly explain the main
ingredient of the proof of Theorem~\ref{thm:main}.

The Hilbertianity property can be reformulated in terms of fields
and places as follows (see Lemma~\ref{lem:placepolynomials}): Let
$t$ be a transcendental element over a field $K$. Then $K$ is
Hilbertian if and only if the following property
\eqref{property:place} holds for any finite separable extension
$F/K(t)$ and nonzero $p(T)\in K[T]$.
\begin{enumerate}
\renewcommand{\theenumi}{*}\renewcommand{\labelenumi}{(\theenumi)}
\item
\label{property:place} There exists a $K$-place $\psi$ of $F$ such
that $a=\psi(t)\in K$, $p(a)\neq 0$, and the degree of $\psi$ equals
to the degree $[F:K(t)]$.
\end{enumerate}

Here a \textbf{$K$-place} of a function field $F/K$ is a place
$\phi$ of $F$ such that $\phi(x) = x$ for all $x\in K$. The degree
of $\phi$ is defined to be $\deg \phi = [N:K]$, where $N$ is the
residue field of $F$ under $\phi$.

In \cite{FriedVolklein1992} Fried and V\"olklein introduce the class
of Regular-Galois-Hilbertian fields -- An \textbf{RG-Hilbertian}
field is a field which satisfies \eqref{property:place} for any
finite Galois $F/K(t)$ for which $F$ is regular over $K$ and nonzero
$p(T)\in K[T]$. This class of fields is important in the context of
the inverse Galois problem. For example, considering a PAC field $K$
of characteristic $0$, they showed that $K$ is RG-Hilbertian if and
only if any finite group occurs as a Galois group over $K$ and that
$K$ is Hilbertian if and only if $K$ is $\omega$-free (i.e.\ any
finite embedding problem has a proper solution). These results are
generalized for a field with an arbitrary characteristic by Pop
\cite{Pop1996}.

Using these group theoretic characterizations of Hilbertianity over
a PAC field, Fried and V\"olklein give an example of a PAC field
which is non-Hilbertian but is RG-Hilbertian, by constructing a
projective profinite group having any finite group as a quotient,
but some finite embedding problem is not solvable.

In \cite{DebesHaran1999} D\`ebes and Haran construct some concrete
new examples of non-Hilbertian RG-Hilbertian fields, which, in
contrast to Fried-V\"olklein examples, are not PAC, and are even
quite small over $\bbQ$ in a certain sense. Also, they exhibit other
variants of Hilbertianity which divide into two kinds. The first is
consisted on the R-Hilbertian and G-Hilbertian fields, which satisfy
\eqref{property:place} for any regular, resp.\ Galois, $F/K(t)$.

The second kind comes from the following characterization of
Hilbertian fields. A necessary and sufficient condition for
Hilbertianity is that for any irreducible
$f_1(T,X),\ldots,f_r(T,X)\in K[T,X]$ that are separable and of
degree $>1$ in $X$ and for any nonzero $p(T)\in K[T]$ there exists
$a\in K$ for which $p(a)\neq 0$ and no $f_i(a,X)$ has a root in $K$
\cite[Lemma 13.1.2 and Proposition 13.2.2]{FriedJarden2005}. Then
the class of \emph{Mordellian} fields is defined -- a field $K$ is
Mordellian if the above specialization property holds for one
polynomial (i.e.\ $r=1$). Similarly to the above are defined
R-Mordellian, RG-Mordellian, and G-Mordellian fields.

D\`ebes and Haran then sum up (using the following nice diagram) the
connections between all the variants. Also none of the converses to
(2), (3), (4), (5), and (6) in the diagram holds \cite[Theorem
5.1]{DebesHaran1999}. In case $K$ is PAC, using a sophisticated
group theoretic construction, D\`ebes and Haran show that the
converse of (1) holds, but for an arbitrary $K$ they simply say
``\textit{We do not know whether the converse of (1) holds in
general...}'' Theorem~\ref{thm:main} then completes the job by
showing that the converse of (1) always holds.

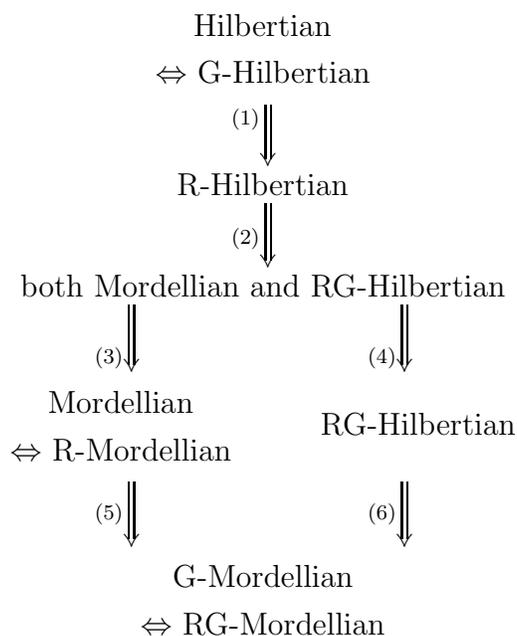
\def\entry#1{\hbox to 1.3truecm{\hss #1 \hss}}
\[
\xymatrix{%
\rm {\begin{array}{c}
    \entry{Hilbertian}\\ \entry{$\Leftrightarrow$
    G-Hilbertian}
\end{array}}
    \ar@{=>}[d]_{(1)}\\
{\entry{R-Hilbertian}}
    \ar@{=>}[d]_{(2)}\\
{\entry{both Mordellian and RG-Hilbertian}}
    \ar@<-10ex>@{=>}[d]_{(3)}
    \ar@<10ex>@{=>}[d]_{(4)}\\
{\begin{array}{c}
    \entry{Mordellian}\\
    \entry{$\Leftrightarrow$ R-Mordellian}
\end{array}}
\qquad\qquad\qquad \entry{RG-Hilbertian}
    \ar@<-10ex>@{=>}[d]_{(5)}
    \ar@<10ex>@{=>}[d]_{(6)}\\
{\begin{array}{c} \entry{G-Mordellian}\\ \entry{$\Leftrightarrow$
RG-Mordellian}\end{array}}
}%
\]

We conclude the introduction by a brief survey of the proof of
Theorem~\ref{thm:main}. It is well known that it suffices to verify
\eqref{property:place} in case $F/K(t)$ splits, i.e., we can assume
$F = F_0L$, where $F_0/K(t)$ is regular and $L/K$ Galois (see
Lemma~\ref{lem:Reduction}). A simple observation is that an
irreducible specialization of $F_0$ gives an irreducible
specialization of $F$ if and only if the residue field of $F_0$ is
linearly disjoint from $L$ over $K$.

The main argument is to consider many copies of $F_0/K(t)$ and then
to simultaneously irreducibly specialize all of them. Then if we
have enough copies, at least one of the specializations is `good,'
i.e., its residue field would be linearly disjoint from $L$.

\noindent \textbf{Notation.} Throughout the paper we use $E,K,L,F$
to denote fields, $T,X$ for variables, and $t$ for a transcendental
element over $K$. Bold face letters denote tuples, e.g., $\bfT =
(T_1,\ldots,T_r)$ (resp., $\bft=(t_1,\ldots,t_r)$) denotes a tuple
of variables (resp., transcendental elements). As above, we say that
an extension $F/K(\bft)$ is regular, if $F$ is regular over $K$.

\section{Auxiliary Lemmas}
\begin{lemma}\label{lem:placepolynomials}
A field $K$ is Hilbertian if and only if \eqref{property:place}
holds for every finite separable extension $F/K(t)$ and nonzero
$p(T)\in K[T]$.
\end{lemma}

\begin{proof}
\cite[Lemma~13.1.1]{FriedJarden2005} implies that a Hilbertian
field $K$ satisfies \eqref{property:place}.

Conversely, let  $f(T,X)\in K[T,X]$ be an irreducible polynomial
that is separable in $X$ and let $0\neq p(T)\in K[T]$. Let $q(T)$ be
the product of $p(T)$ with the leading coefficient of $f(T,X)$ and
its discriminant (regarding $f$ as a  polynomial in $X$ over
$K(T)$). Let $\psi$ be the corresponding $K$-place that
\eqref{property:place} gives for $F/K(t)$ and $q(T)$, where $F =
K(t)[X]/(f(t,X))$.

Then the residue field $N$ of $F$ under $\psi$ is generated by a
root of $f(a,X)$ \cite[Remark~6.1.7]{FriedJarden2005}. Thus
$[F:K(t)]=[N:K]$ implies that $f(a,X)$ is irreducible.
\end{proof}

The following observation gives a sufficient condition for a
polynomial to have an irreducible specialization in terms of a place
of a regular extension having certain properties.

\begin{lemma}\label{lem:Reduction}
Let $f(T,X)\in K[T,X]$ be an irreducible polynomial that is
separable in $X$. Then there exists a nonzero $p(T)\in K[T]$, a
Galois extension $L/K$, and a separable regular extension $F/K(t)$
such that the following holds. Let $\psi$ be a $K$-place of $F$ with
residue field $N$. Assume that $a=\psi(t)\in K$, $p(a)\neq 0$,
$[N:K] = [F:K(t)]$, and $N$ is linearly disjoint from $L$ over $K$.
Then $f(a,X)$ is irreducible.
\end{lemma}

\begin{proof}
Let $x$ be a root of $f(t,X)$ in a separable closure of $K(t)$. By
\cite[Lemma 13.1.3]{FriedJarden2005} there exist fields $F$ and $L$
such that $F/K$ is regular and $t\in F$, $L/K$ is Galois, $x\in FL$
and $FL/K(t)$ is Galois. Let $E= FL$ and let $p(t)$ be the product
of the leading coefficient of $f(t,X)$ and its discriminant as a
polynomial in $X$.

It suffices to find a $K$-place $\phi$ of $E$ such that
$a=\phi(t)\in K$, $p(a)\neq 0$, $\deg\phi = [E:K(t)]$ (w.r.t.\
$E/K$). Indeed, assume $\phi$ is such a place and let $M$ denote the
residue field of $E$. Since $p(a)\neq 0$ we have that $b=\phi(x)$ is
finite. Hence the residue field of $K(t,x)$ under $\phi$ is $K(b)$.
Since $f(a,b)=0$ we have
\begin{eqnarray*}
\deg f(a,X)&\geq& [K(b):K] = \frac{[M:K]}{[M:K(b)]} =
\frac{[E:K(t)]}{[M:K(b)]}\geq\frac{[E:K(t)]}{[E:K(t,x)]}\\
&=& [K(t,x):K(t)] = \deg_X f(t,X)=\deg f(a,X).
\end{eqnarray*}
Therefore $\deg f(a,X) = [K(b):K]$ which implies that $f(a,X)$ is
irreducible, as needed.

\[
\xymatrix{%
F\ar@{-}[d] \ar@{-}[r]
    &E\ar@{-}[d]\ar@{-}[dl]|{K(t,x)}\\
K(t)\ar@{-}[r]
    &L(t)
}%
\qquad \mathop{\dashrightarrow}^\phi \qquad
\xymatrix{%
N\ar@{-}[d] \ar@{-}[r]
    &M\ar@{-}[d]\ar@{-}[dl]|{K(b)}\\
K\ar@{-}[r]
    &L
}%
\]

Let $\psi$ be the $K$-place of $F$ given by the assumption. Extend
$\psi$ to an $L$-place $\phi$ of $E$. Let $M$, $N$ be the respective
residue fields of $E$, $F$ under $\phi$. Then as $E=FL$ and $\phi$
is an $L$-place we have that $M = NL$. Since $N$ and $L$ are
linearly disjoint over $K$, $F$ and $L(t)$ are linearly disjoint
over $K(t)$, and $\deg\psi=[F:K(t)]$ it follows that
\[
[M:K]=[N:K][L:K]=[F:K(t)][L:K]=[E:K(t)].
\]
Finally, since $\psi(t) = \phi(t)$, we have $p(\phi(t))\neq 0$, and
thus the assertion follows.
\end{proof}

A similar argument gives the following result.
\begin{lemma}\label{lem:lineardisjointness}
Let $F_1,\ldots, F_r$ be linearly disjoint separable extensions of a
field $E$ and let $F = F_1\cdots F_r$. Let $\phi$ be a place of
$F/E$ with a residue field extension $N/K$ and of degree $\deg\phi =
[F:E]$. Let $N_i$ be the residue field of $E_i$ under $\phi$, for
each $i=1,\ldots,r$. Then $[N_i:K]=[F_i:E]$ and $N_1,\ldots, N_r$
are linearly disjoint over $K$.
\end{lemma}

\begin{proof}
Let $F_0$ be a subextension of $F/E$ with residue field $N_0$. As
$[F:E]=[N:K]$ we have
\[
[F_0:E] = [F:E]/[F:F_0] = [N:K]/[F:F_0] \leq [N:K]/[N:N_0] =
[N_0:K]\leq [F_0:E],
\]
and hence $[N_0:K]=[F_0:E]$. In particular, for $F_0=F_i$ we get
$[N_i:K]=[F_i:E]$. Next take $F_0=F$. Then $N_0 = N = N_1\cdots
N_r$, and we have
\[
[N_1 \cdots N_r : K]=[N:K] = [F:E] = [F_1:E]\cdots [F_r:E] =
[N_1:K]\cdots[N_r:K],
\]
which implies that $N_1,\ldots,N_r$ are linearly disjoint over $K$.
\end{proof}

The following well known consequence of Bertini-Noether lemma and
Matsusaka-Zariski theorem reduces the transcendence degree of a
regular extension to $1$ (cf. proof of \cite[Proposition
13.2.1]{FriedJarden2005}). For the sake of completeness the proof is
included below.

\begin{lemma}\label{lem:fromrto1}
Let $r\geq 2$, let $(t,\bft) = (t,t_1, \ldots, t_r)$ be an
$(r+1)$-tuple of algebraically independent transcendental elements
over an infinite field $K$, and let $E/K(\bft)$ be a finite regular
separable extension. Then there exist $\alpha_i,\beta_i\in K$,
$\beta_i\neq 0$, $i=2,\ldots,r$ for which the specialization
$\bft\mapsto(t,\alpha_2 + \beta_2t ,\ldots, \alpha_r + \beta_r t)$
extends to a $K$-place $\phi$ of $E$ with a regular residue field
extension of degree $\deg\phi = [E:K(\bft)]$.
\end{lemma}

\begin{proof}
Let $x\in E$ be integral over $K[\bft]$ such that $E=K(\bft,x)$. Let
$f(\bfT,X)\in K[\bfT,X]$ be the absolutely irreducible, separable
and monic in $X$ polynomial for which $f(\bft,x)=0$ and let
$p(\bfT)$ be its discriminant.

Take two variables $U,V$. Matsusaka-Zariski Theorem implies that
$f(T_1,\ldots, T_{r-1}, U + V T_1,X)$ is irreducible in the ring
$\Lgal [T_1,\ldots,T_{r-1} , X]$, where $\Lgal$ is an algebraic
closure of $K(U,V)$ \cite[Proposition 10.5.4]{FriedJarden2005}. By
Bertini-Neother Lemma there exists a nonzero $c(U,V)\in K[U,V]$ such
that for any $\alpha_r,\beta_r\in K$ satisfying
$c(\alpha_r,\beta_r)\neq 0$ the polynomial $f(T_1,\ldots,T_{r-1},
\alpha_r + \beta_r T_1 , X)$ remains absolutely irreducible over $K$
\cite[Proposition 10.4.2]{FriedJarden2005}. Since $K$ is infinite,
there exist $\alpha_r,\beta_r\in K$, $\beta_r\neq 0$ such that
$c(\alpha_r,\beta_r)\neq 0$ and $p(T_1,\ldots, T_{r-1}, \alpha_r +
\beta_r T_{1})\neq 0$. Induction on $r$ yields $\alpha_i,\beta_i\in
K$, $\beta_i\neq 0$, $i=2,\ldots, r$ such that $g(T,X) = f(T,
\alpha_2 + \beta_2 T,\ldots, \alpha_r + \beta_r T,X)$ is an
absolutely irreducible polynomial and $q(T) = p(T, \alpha_2 +
\beta_2 T, \ldots, \alpha_r + \beta_r T)\neq 0$.

To conclude the proof, extend the specialization
$\bft\mapsto(t,\alpha_2 + \beta_2t ,\ldots, \alpha_r + \beta_r t)$
to a $K$-place $\phi$ of $E$ with a residue field extension
$E'/K(t)$ \cite[Lemma 2.2.7]{FriedJarden2005}. Then $E' = K(t,x')$
where $x'=\phi(x)$ is a root of $g(t,X)$; hence $E'$ is regular over
$K$ and
\[
\deg\phi = [E':K(t)] = \deg_X g(t,X)=\deg_X f(\bft,X)=[E:K(\bft)],
\]
as needed.
\end{proof}

For a field extension $L/K$ we set $s(L/K)$ to be the number of
subextensions $K\subseteq L_0 \subseteq L$. Note that if $L/K$ is
finite and separable, then $s(L/K)$ is also finite.

\begin{lemma}\label{lem:toomanylinearlydisjoint}
Let $L/K$ be a finite separable extension with Galois closure $E/K$.
Let $r\geq s(E/K)$ and let  $N_1,\cdots, N_r$ be linearly disjoint extensions of $K$. Then there
exists $i\in \{1,\ldots,r\}$ for which $N_i$ is linearly disjoint
from $L$ over $K$.
\end{lemma}

\begin{proof}
It suffices to show that there exists $i\in \{1,\ldots,r\}$ for
which $N_i$ is linearly disjoint from $E$ over $K$. Let $E_i =
N_i\cap E$. As $E/K$ is Galois, $N_i$ and $E$ are linearly disjoint
if and only if $E_i=K$. Assume thus that $E_i\neq K$ for all $i$.
Since $r>s(E/K)-1$, the pigeonhole principle gives $i\neq j$ for
which $E_i=E_j$. But $E_i\cap E_j\subseteq N_i\cap N_j = K$, which
implies that $E_i=K$, a  contradiction.
\end{proof}

\section{Proof of Theorem~\ref{thm:main}}
We assume that for any absolutely irreducible polynomial $g(T,X)\in
K[T,X]$ and nonzero $p(T)$ there exists $a\in K$ such that $p(a)\neq
0$ and $g(a,X)$ is irreducible. This assumption implies
\eqref{property:place} for regular $F/K$ and nonzero $p(T)$.

Let $F/K(t)$ be a separable extension of degree $n$ with $F/K$
regular and let $L/K$ be a Galois extension. By
Lemma~\ref{lem:Reduction}, it suffices to show that there exists a
$K$-place $\psi$ of $F$ satisfying
\begin{eqnarray}
&&a=\psi(t)\in K\label{eq:ainK}\\
&&p(a)\neq 0\label{eq:pneq0}\\
&&[N:K]=n\label{eqLdegequaln}\\
&&\hbox{$N$ and $L$ are linearly disjoint over
$K$,}\label{eq:LinDis}
\end{eqnarray}
where $N$ is the residue field of $F$ under $\psi$.

Let $r \geq s(L/K)$ (recall that $s(L/K)$ is the number of
subextensions of $L/K$). Take $r$ algebraically disjoint copies of
$F/K(t)$, that is to say, consider an $r$-tuple
$\bft=(t_1,\ldots,t_r)$ of algebraically independent transcendental
elements and for each $i=1,\ldots, r$ consider an extension
$F_i/K(t_i)$ and a $K$-isomorphism $\nu_i\colon F\to F_i$ under
which $t$ maps to $t_i$. Let $E_i = F_i K(\bft)$ and $E = E_1\cdots
E_r$. Then $E_1,\ldots, E_r$ are linearly disjoint over $K(\bft)$
and we have
\begin{equation*}
[E_i:K(\bft)] = [F_i:K(t_i)]=n,
\end{equation*}
for $i=1,\ldots,r$, and thus $n^r = [E_1:K(\bft)]\cdots
[E_r:K(\bft)]=[E:K(\bft)]$.

Lemma~\ref{lem:fromrto1} gives $\alpha_i,\beta_i\in K$, $\beta_i\neq
0$ for which the specialization $\bft \mapsto (t_0,\alpha_2 +
\beta_2t_0 ,\ldots, \alpha_r + \beta_r t_0)$ (with transcendental
element $t_0$) extends to a $K$-place $\phi$ of $E/K(\bft)$ with a
regular residue field $E'/K(t_0)$ of degree $\deg\phi =
[E:K(\bft)]=n^r$. Let $E'_i$ be the residue field of $E_i$ under
$\phi$, $i=1,\ldots, r$.

Note that the set $A_0=\{a_0\in K \mid \exists 1\leq i\leq r \mbox{
such that } p(\alpha_i + \beta_i a_0)=0\}$ is finite, since
$\beta_i\neq 0$, and hence $q(T)=\prod_{a_0\in A_0}(T-a_0)$ is a
polynomial. Therefore we can apply \eqref{property:place} for the
regular extension $E'/K(t_0)$ and $q(T)\in K[T]$ to get a place
$\psi'$ of $E'$ of degree $n^r$ such that $a_0=\psi'(t_0)\in K$ and
$q(a_0)\neq 0$.

Let $N_i$ be the residue field of $E'_i$ under $\psi'$,
$i=1,\ldots,r$. In this setting, Lemma~\ref{lem:lineardisjointness}
asserts that $[N_i:K] = [E'_i:K(t_0)] = [E_i:K(\bft)]=n$ and
$N_1,\ldots,N_r$ are linearly disjoint over $K$. Since $r\geq
s(L/K)$, Lemma~\ref{lem:toomanylinearlydisjoint} gives $i\in \{
1,\ldots, r\}$ for which $N_i$ and $L$ are linearly disjoint over
$K$.

\[
\xymatrix@R=2pt{%
F\ar@{-}[ddd]\ar[r]_{\nu_i}\ar@/^10pt/[rrrr]^{\psi}
    & F_i\ar@{-}[ddd]\ar@{^(->}[r]
        &E_i\ar@{-}[ddd]\ar[r]_{\phi_i}
            &E'_i\ar@{-}[ddd]\ar[r]_{\psi'}
                &N_i\ar@{-}[ddd]
\\
\\
\\
K(t)\ar[r]
    &K(\bft)\ar@{^(->}[r]
        &K(\bft)\ar@{->}[r]
            &K(t_0)\ar[r]
                &K
\\
t\ar@{|->}[r]^{\nu_i}
    &t_i\ar@{|->}[r]
        &t_i\ar@{|->}[r]
            &\alpha_i+\beta_i t_0\ar@{|->}[r]
                &\alpha_i+\beta_i a_0
}%
\]

To conclude the proof set $\psi = \psi'\phi_i\nu_i$, where $\phi_i =
\phi|_{E_i}$ is the restriction of $\phi$ to $E_i$. We have
\[
a = \psi(t) = \psi'(\phi(t_i)) = \psi'(\alpha_i + \beta_i t_0) =
\alpha_i + \beta_i a_0\in K,
\]
and hence \eqref{eq:ainK} holds. The residue field of $F$ under
$\psi$ is $N_i$; this proves \eqref{eqLdegequaln} and
\eqref{eq:LinDis}. Recall that $q(a_0)\neq 0$ implies that $p(a)\neq
0$, hence we have \eqref{eq:pneq0} and the proof of
Theorem~\ref{thm:main} is completed.

\begin{remark}[Hilbert sets]
Let $f(\bfT,X)\in K[\bfT,X]$ be an irreducible polynomial that is
separable in $X$ and $p(\bfT)\in K[T]$. Then the corresponding
\textbf{Hilbert set} is defined to be
\[
H_K(f;p) =\{\bfa\in K^r \mid f(\bfa,X) \mbox{ is irreducible and }
q(\bfa)\neq 0 \}\subseteq K^r.
\]

The proof of Theorem~\ref{thm:main} gives the following assertion.
Let $f(T,X)=f_0(T)X^d+ \cdots \in K[T,X]$ be an irreducible
polynomial that is separable in $X$ and let $p(T)\in K[T]$ be
nonzero. Then there exists an absolutely irreducible polynomial
$g(T_1,\ldots, T_r,X)\in K[\bfT,X]$ (for any sufficiently large $r$)
and $q(\bfT)\in K[\bfT]$ such that the following holds. For any
$\bfa\in H_K(g;q)$ there exists $i\in \{1,\ldots,r\}$ such that
$a_i\in H_{K}(f;p)$.

(The above assertion follows from the proof of
Theorem~\ref{thm:main} by taking $g(\bfT,X)$ such that a root of
$g(\bft,X)$ generates $E/K(\bft)$ and $q(\bfT)=\prod_{i=1}^r
p(T_i)f_0(T_i)$.)
\end{remark}

The above remark leads to a slightly finer question than
D\`ebes-Haran question, which might be of interest:
\begin{problem}[cf.\ {\cite[Problem
13.1.5]{FriedJarden2005}}] Let $f(T,X)\in K[T,X]$ be an irreducible
polynomial, separable in $X$ and let $p(T)\in K[T]$ be nonzero. Does
there exist an absolutely irreducible $g(T,X)$ and nonzero $q(T)$
such that $H_K(g;q)\subseteq H_K(f;p)$?
\end{problem}

\section{Small Extensions of Hilbertian Fields}
\begin{definition}
An extension $M/K$ is called \textbf{small} if $M/K$ is Galois and
for every positive integer $n$ there exists finitely many
subextensions $L$ of $M/K$ of degree $\leq n$.
\end{definition}

\begin{exampleplain}
If $\gal(M/K)$ is finitely generated, then $\gal(M/K)$ is small
\cite[Lemma~16.10.2]{FriedJarden2005}.
\end{exampleplain}

\begin{exampleplain}
If $\gal(M/K) \cong \prod_p \bbZ_p^p$, where $p$ runs over all
primes and $\bbZ_p$ is the $p$-adic group, then $M/K$ is small and
$\gal(M/K)$ is not finitely generated
\cite[Example~16.10.4]{FriedJarden2005}.
\end{exampleplain}

We reprove \cite[Proposition~16.11.1]{FriedJarden2005} using
Theorem~\ref{thm:main}.
\begin{theorem}
Let $K$ be a Hilbertian field and $M/K$ a small extension. Then $M$
is Hilbertian.
\end{theorem}

\begin{proof}
Let $f(T,X)\in M[T,X]$ be an absolutely irreducible polynomial that
is separable and of degree $n$ in $X$. Let $K'$ be a finite
subextension of $M/K$ such that $f(T,X)\in K'[T,X]$. Then $K'$ is
Hilbertian \cite[Corollary~12.2.3]{FriedJarden2005}. Evidently
$M/K'$ is also small. Let $r> \#$ of subextensions of $M/K'$ of
degree $\leq n$.

Let $F = K'(t)[X]/(f(t,X))$. Then $F$ is regular over $K'$. Now we
proceed as in the proof of Theorem~\ref{thm:main}: Take $r$ copies
$F_1/K'(t_1),\ldots, F_r/K'(t_r)$ of $F/K(t)$. From the
Hilbertianity of $K'$ we get a specialization $(t_1,\ldots, t_r)
\mapsto (a_1, \ldots, a_r) \in K'^r$ such that the residue fields
$N_i$ of $F_i$ are linearly disjoint, and $[N_i:K] = [F_i : K(t_i)]
= n$.

Now, if there exists $1\leq i\leq r$ such that $N_i\cap M = K'$,
then $[N_i M : M] = [N_i : K'] = n$. But $N_i M$ is the residue
field of $F_iM$; so we are done.

Assume that $L_i := N_i\cap M \neq K'$ for all $1\leq i\leq r$.
Then, since $N_i\cap N_j = K'$ we have $L_i\cap L_j = K'$. In
particular, $L_1, \ldots, L_r$ are distinct subextensions of $M/K'$
of degrees $\leq n$. This contradicts the choice of $r$.
\end{proof}

\begin{remark}
Actually a stronger statement than we proved is true, namely $M$ is
Hilbertian over $K$. Indeed, the fact that $K'$ is not only
Hilbertian, but also Hilbertian over $K$
\cite[Corollary~12.2.3]{FriedJarden2005} implies that we can choose
$a_i$ to be in $K$.

This stronger assertion is also proved in
\cite[Proposition~16.11.1]{FriedJarden2005}.
\end{remark}

\bibliographystyle{amsplain}


\end{document}